\documentclass[12pt]{article}
\usepackage{graphics}
\usepackage{graphicx}
\usepackage[mathscr]{eucal}
\usepackage{epic,eepicemu,eepic}
\usepackage{amsmath,amsxtra,amssymb,latexsym,amscd,amsthm,amsfonts}
\usepackage[left=3cm,right=2.5cm,top=2.5cm,bottom=2.5cm]{geometry}
\newtheorem{thm}{Theorem}
\newtheorem{cor}[thm]{Corollary}

\newtheorem{lem}[thm]{Lemma}

\DeclareMathOperator{\Vol}{Vol}
\DeclareMathOperator{\const}{const}
\begin{document}
\baselineskip=17pt
\title{\bf Bernstein type theorem for entire weighted minimal graphs in $\Bbb G^n\times\Bbb R.$\\
 \small{}}
\author{\bf  Doan The Hieu\\
 Departement of Mathematics\\
 College of Education, Hue University, Hue, Vietnam\\
 dthehieu@yahoo.com\\
{\bf Tran Le Nam} \\
 Departement of Mathematics\\
 Dong Thap University, Dong Thap, Vietnam\\
 lenamdongthapmuoi@gmail.com}
 \maketitle
\begin{abstract}
Based on a calibration argument, we prove a Bernstein type theorem for entire minimal graphs over  Gauss space $\Bbb G^n$ by a  simple proof.
 \end{abstract}
 \vskip 1cm
\section{Introduction}
The classical Bernstein theorem asserts that an entire minimal graph over $\Bbb R^2$ is a plane (see \cite{ber}, \cite {hei}, \cite {oss1}). The theorem has been generalized to dimensions $n\le 7$ (see \cite{gio1}, \cite{alm}, \cite{sim}) and has been proved to be false in dimensions $n\ge 8$ (see \cite{bom}). Entire minimal graphs over $\Bbb R^n,\ n\ge 8,$ that are not hyperplanes, were found by Bombieri, De Giorgi and Guisti in \cite {bom}.
For arbitrary dimensions, the theorem is settled under some hypotheses on the growth of the minimal graph.  The theorem is also studied in product spaces $M^n\times \Bbb R$ (see \cite{esp}), where $M$ is a Riemannian manifold,  as well as for self-similar shrinkers. Ecker and Huisken \cite{eck} proved a Bernstein type theorem for self-similar shrinkers that are entire graph and have at most polynomial volume growth. Later the condition on volume growth is removed by Wang \cite{wan}.

Self-similar shrinkers, a simplest solution to mean curvature flow that satisfies the following equation
$$H=\frac12\langle \vec{x},{\bf n}\rangle,$$
are just minimal hypersurfaces in $\Bbb R^{n+1}$ under the conformally changed metric $g_{ij}=\exp(-|\overrightarrow{x}|/2n)\delta_{ij}$
(see \cite{CM1}) or a special example of weighted minimal hypersurfaces in $\Bbb R^{n+1}$ with density $e^{-\frac{|x|^2}4},$ a modified version of Gauss space.
Gauss space $\Bbb G^{n+1},$ Euclidean space $\Bbb R^{n+1}$ with Gaussian probability density $e^{-f}=(2\pi)^{-\frac {n+1}2}e^{-\frac{|x|^2}2},$ is a typical example of a manifold with density and very interesting to probabilists.

It should be mentioned that the Bernstein type theorem for self-similar shrinkers can be applied for $f$-minimal graphs in $\Bbb G^{n+1}=\Bbb G^n\times\Bbb G^1.$

Motivated by these works, we study a Bernstein type theorem for entire weighted minimal graphs in the product space $\Bbb G^n\times \Bbb R,$ that is $\Bbb R^{n+1}$ with mixed Gaussian-Euclidean density.

Based on a calibration argument, we prove a volume growth estimate of a weighted minimal graph in the product space $\Bbb G^n\times \Bbb R$ and the  theorem follows easily.
 Our proof is quite simple, elementary and holds for arbitrary dimensions without any additional conditions. The proof is also adapted for the case in which Gauss space is replaced by $\Bbb R^n$ with a radial density and finite weighted volume.


\section{Weighted minimal surfaces in $\Bbb G^n\times\Bbb R$}

A manifold with density, also called a weighted manifold, is a Riemannian manifold with a positive function $e^{-f}$ used to weight both volume and perimeter area. The weighted mean curvature of hypersufaces on a such manifold
 is defined as follows
\begin{equation}  \label {Hf}
    H_{f}=H+\langle \nabla f,{\bf n}\rangle,\end{equation}
where $H$ is the Euclidean mean curvature and ${\bf n}$ is the
normal vector field of the hypersurface.  A hypersurface with $H_f=0$ is called a weighted minimal hypersurface or an $f$-minimal hypersurface. For more details about manifolds with density, we refer the reader to \cite{mo1}, \cite{mo2}, \cite{mo5}.

Let $S$ be a regular hypersurface in $\Bbb G^{n}\times \Bbb R,$ \ $M$ be a point on $S$ and $\rho$ denote the projection onto $x_{n+1}$-axis. The following lemma shows a geometric meaning of $\langle \nabla f,{\bf n}\rangle.$

\begin{lem}\label{21}
$$|\langle\nabla {\varphi},{\bf n}\rangle|=d(\rho(M), T_MS).$$
\end{lem}
\begin{proof}
Let ${\bf n}(a_1,a_2,\ldots, a_{n+1})$ be a normal vector of $S$ at $M.$ Then an equation of $T_MS$ is
$$\sum_{i=1}^{n+1} a_ix_i+d=0.$$
Therefore,
$$d(\rho(M),T_MS)=|a_{n+1}x_{n+1}+d|=|\langle(x_1, x_2,\ldots,x_n, 0),(a_1,a_2,\ldots,a_n,a_{n+1})\rangle|=|\langle\nabla {f},{\bf n}\rangle|.$$
\end{proof}

Below are simple examples of constant weighted mean curvatures and weighted minimal surfaces in $\mathbb{G}^2\times \mathbb{R}.$
\begin{enumerate}
\item
 Planes parallel to the $z$-axis have constant weighted mean curvature and planes containing the $z$-axis are weighted minimal.
 \item Planes $z=a$ are weighted minimal.
\item Right circular cylinders about $z$-axis have  constant weighted mean curvature and the one of radius 1 is weighted minimal.
\item Consider the deformation determined by a family of parametric minimal surfaces given by
\begin{align}X_\theta(u,v)&=(x(u,v), y(u,v), z(u,v));\\
x(u,v)&=\cos\theta\sinh v\sin u+\sin\theta\cosh v\cos u,\\
y(u,v)&=-\cos\theta\sinh v\cos u+\sin\theta\cosh v\sin u,\\
z(u,v)&=u\cos\theta+v\sin\theta;
\end{align}
where $-\pi<u\le\pi,\ -\infty<v<\infty$ and the defomation parameter $-\pi<\theta\le\pi.$
A direct computation shows that $X_\theta$ is minimal with the normal  vector field
$$N=\left(\frac{\cos u}{\cosh v},  \frac{\sin u}{\cosh v}, -\frac{\sinh u}{\cosh v}\right).$$
Since $\nabla f=(x, y, 0),$\ $\langle\nabla f, N\rangle=\sin\theta,$\ $X_{\pi/2}$ is the catenoid while $X_0$ is the helicoid, it follows that
\begin{enumerate}
\item In $\Bbb G^2\times \Bbb R,$\ $X_\theta$ has constant weighted curvature.
\item The helicoid $X_0$ is a ruled weighted minimal  surface in $\Bbb G^2\times \Bbb R.$
\item The catenoid $X_{\pi/2}$ has constant weighted  curvature 1 in $\Bbb G^2\times \Bbb R.$
    \end{enumerate}    \end{enumerate}
\section{A Bernstein type theorem for weighted minimal graphs in $\Bbb G^n\times\Bbb R$}

\subsection{Minimality of hyperplanes in  $\Bbb G^n\times(\Bbb R, e^{-h})$}

Consider the product space $\Bbb G^n\times(\Bbb R, e^{-h})$ with product density $e^{-(f+h)}.$
 A point in $\Bbb G^n\times(\Bbb R, e^{-h})$ can be written as $({\bf x},x_{n+1}),$ where ${\bf x}=(x_1, x_2,\ldots, x_n)\in \Bbb R^n.$ An equation of a non-vertical hyperplane in $\Bbb G^n\times(\Bbb R, e^{-h})=\Bbb R^{n+1}$ is of the form
 $$\Sigma_{i=1}^na_ix_i+x_{n+1}+c=0, \ \ \ c, a_1, a_2, \ldots, a_n\in\Bbb R.$$
 A direct computation shows that $\langle \nabla(f+h),{\bf n}\rangle=0$ if and only if
 $$ \Sigma_{i=1}^na_ix_i+h'(x_{n+1})=0.$$
Thus,
 \begin{enumerate}
 \item the plane is weighted minimal if and only if $h'(-c)=0,$
 \item the plane is weighted minimal and non-horizontal if and only if $h(x_{n+1})=x_{n+1}^2/2+cx_{n+1}+b,$ where $b\in \Bbb R$ is a constant.
 \end{enumerate}
If $h$ is monotone, any hyperplane is not weighted minimal and a  Bernstein type theorem does not exist for this case. Since weighted minimality of a hypersurface in $\Bbb G^{n+1}$ is equivalent to that in  $\Bbb G^n\times(\Bbb R, e^{-h}),$ where $h(x_{n+1})=x_{n+1}^2/2+cx_{n+1}+b,$ the Bernstein type theorem   proved by Ecker-Huisken and Wang (\cite{eck},  \cite{wan}) for self-similar shrinkers is adapted for weighted minimal graphs in the later case. The following example shows that in $\Bbb G^n\times(\Bbb R_+, e^{-h}),$  where $\Bbb R_+=\{x\in\Bbb R: x\ge 0\}$ and  $h(z)=z^2-\ln{\sqrt{1+4z}},$ there exist both hyperplanar and non-hyperplanar entire weighted minimal graphs.

Consider the graph of the function $z=u(x,y)=x^2$ over $\Bbb G^2$ in $\Bbb G^2\times(\Bbb R, e^{-h}).$  A direct computation yields
$$H_{(f+g)}=\frac{1}{(1+4z)^{3/2}}-\frac{2z+h'(z)}{2\sqrt{1+4z}}=0.$$
Moreover, the horizontal planes $z=(1+\sqrt{17})/8$ are also weighted minimal.

 \subsection{Entire weighted minimal graph in $\Bbb G^n\times\Bbb R$}
This subsection considers the case where $h=\const.$ and we can assume that $h=1.$
The space $\Bbb G^n\times\Bbb R$ is just $\Bbb R^{n+1}=\Bbb R^n\times\Bbb R$ endowed with the Euclidean-Gaussian density
$$e^{-f}=(2\pi)^{-\frac n2}e^{-\frac{|{\bf x}|^2}2}.$$
Denote by:
\begin{itemize}
       \item  $B^{n+1}(p,R)$ the $(n+1)$-ball in $\Bbb G^n\times \Bbb R$ with center $p$ and radius $R,$
       \item $B^{n}(p,R)$ the $n$-ball in $\Bbb G^n$ with center $p$ and radius $R,$
       \item $S^{n+1}(p,R)$ the $(n+1)$-sphere in $\Bbb G^n\times \Bbb R$ with center $p$ and radius $R,$
       \item $S^{n}(p,R)$  the $n$-sphere in $\Bbb G^n$ with center $p$ and radius $R,$
       \item $S^{n}_+(p,R)$  the upper half of $S^{n}(p,R),$
            \item $B^{n+1}_{+}(p,R)$  the upper half of $B^{n+1}(p,R),$
     \end{itemize}
Let $\Sigma$ be the weighted minimal graph of a function $u(x_1, x_2, \ldots, x_n)=x_{n+1}$ over $\Bbb G^n$ and let $p$ be the intersection point of $\Sigma$ and $x_{n+1}$-axis, then we have the following area estimates.
\begin{lem}
\begin{equation}\label{eq2}\Vol_f(\Sigma\cap B^{n+1}(p,R))\le \Vol_f(B^n(O, R))+ne^{-R^2}C_nR^{n-1},\end{equation}
where $C_n=\Vol B^n(O,1).$
\end{lem}
\begin{proof}

Let ${\bf n}$ be a unit normal field of $\Sigma$ and consider the smooth extension  of ${\bf n}$ by the translation along $x_{n+1}$-axis, also denoted by ${\bf n}.$

Consider the $n$-differential form defined by
\begin{equation}\label{eq1} w(X_1, X_2,\ldots, X_{n})= \det(X_1, X_2,\ldots, X_{n}, {\bf n}),\end{equation}
 where $X_i,\ i=1,2,\ldots, n$ are smooth vector fields.
 It is not hard to see that:
 \begin{enumerate}
   \item  $|w(X_1, X_2,\ldots, X_{n})|\le 1,$ for every normal vector fields  $X_i,\ i=1,2,\ldots, n$ and the equality holds if and only if
       $X_1, X_2,\ldots, X_{n}$ are tangent to $\Sigma.$
   \item $d(e^{-f}w)=0,$ because $\Sigma$ is weighted minimal.
 \end{enumerate}
Such a differential form is called a weighted calibration that calibrates $\Sigma.$ A simple proof by using Stokes' Theorem proves that $\Sigma$ is weighted area-minimizing, i.e. any compact portion of $\Sigma$ has least area among all surfaces in its homology class (see \cite{hi}).

Now let $\Sigma\cap B^{n+1}(p,R):=\widetilde{\Sigma_R}.$  Since $\widetilde{\Sigma_R}$ is weighted area-minimizing and $\partial\widetilde{\Sigma_R}\subset S^n(p,R),$ we have the following estimate
$$\Vol_f(\widetilde{\Sigma_R})\le\frac 12\Vol_f S^n(p,R).$$
Note that the density does not dependent on the last coordinate, therefore
$$\frac 12\Vol_f S^n(p,R)=\frac 12\Vol_f S^n(O,R)=\Vol_f S_+^n(O,R).$$

Let $\eta$ be the volume form of $\overline{S_+^n(O,R)}$ associated with the outward unit normal defined as in (\ref{eq1}). Its extension  by the translation along $x_{n+1}$-axis in the cylinder $\overline{S^n(O,R)}\times \Bbb R$ is also denoted by ${\eta}.$
By  applying Stokes' theorem together by choosing suitable orientations for objects, we have
$$\begin{aligned}
\Vol_f S_+^n(O,R)&=\int_{S_+^n(O,R)}e^{-f} \eta\\
&=\int_{\overline{B^n}(O,R)}e^{-f} \eta+\int_{\overline{B^{n+1}_{+}(O,R)}} d(e^{-f}\eta) \hskip2.5cm (\text{by Stokes' Theorem})\\
&\le \Vol_f(B^n(O,R))+\int_{\overline{B^n(O,R)}\times [0,R]} d(e^{-f}\eta) \hskip1.2cm (|\eta|\le 1 \ \text{and}\ \ B^{n+1}_{1/2}(O,R)\subset {\cal C}) \\
&= \Vol_f(B^n(O,R))+\int_{S^n(O,R)\times[0,R]}e^{-f}\eta   \hskip2cm (\text{by Stokes' Theorem})\\
&\le \Vol_f(B^n(O,R))+e^{-R^2}\Vol(S^{n-1}(O,R)\times[0,R])\\
&=\Vol_f(B^n(O,R))+ne^{-R^2}C_nR^{n-1}.
\end{aligned}$$
\end{proof}
\begin{cor}
$$\Vol_f(\Sigma)\le 1.$$
\end{cor}
\begin{proof}
Taking the limit of both side of (\ref{eq2}) as $R$ goes to infinity.
\end{proof}

\begin{thm}
The graph $\Sigma$ of a function $u(x_1, x_2, \ldots, x_n)=x_{n+1}$ over $\Bbb G^n$ is weighted minimal if and only if it is a hyperplane $\{x_{n+1}=a\},$ i.e. $u$ is constant.
\end{thm}
\begin{proof}
Of course, a horizontal hyperplane in $\Bbb G^n\times \Bbb R$ is weighted minimal.
Now suppose that $\Sigma$ is weighted minimal.
Let $dV=dx_1\wedge dx_2\wedge\ldots\wedge dx_n.$ We have
$$1\ge\Vol_f(\Sigma)=\int_{\Bbb G^n}e^{-f}\sqrt{1+|\nabla u|^2}dV\ge\int_{\Bbb G^n}e^{-f}dV=\Vol_f(\Bbb G^n)=1.$$
The equality holds if and only if $|\nabla u|^2=0,$ i.e. $u$ is constant.
\end{proof}


\end{document}